\documentclass[11pt]{amsart}

\pdfoutput=1

\usepackage[utf8]{inputenc}
\usepackage[T1]{fontenc}
\usepackage{lmodern}
\usepackage{amsthm, amssymb, amsmath, amsfonts, mathrsfs}

\usepackage{microtype}

\usepackage[pagebackref,colorlinks=true,pdfpagemode=none,urlcolor=blue,
linkcolor=blue,citecolor=blue]{hyperref}

\usepackage{amsmath,amsfonts,amssymb,amsthm}
\usepackage{amsthm, amssymb, amsmath, amsfonts, mathrsfs}

\usepackage{mathrsfs}
\usepackage{MnSymbol}
\usepackage{scalerel} 

\usepackage{color}
\usepackage{accents}

\usepackage{bbm}

\definecolor{labelkey}{gray}{.8}
\definecolor{refkey}{gray}{.8}

\definecolor{darkred}{rgb}{0.9,0.1,0.1}
\definecolor{darkgreen}{rgb}{0,0.5,0}

\newtheorem{theorem}{Theorem}[section]
\newtheorem{lemma}[theorem]{Lemma}

\newtheorem{proposition}[theorem]{Proposition}

\theoremstyle{remark}
\newtheorem{remark}[theorem]{Remark}

\renewenvironment{proof}[1][Proof]{ {\itshape \noindent {#1.}} }{$\Box$
\medskip}

\numberwithin{equation}{section}
\newcommand{\R}{\mathbb{R}}

\newcommand{\Pb}{\mathbb{P}}

\newcommand{\E}{\mathbb{E}}

\newcommand{\F}{\mathcal{F}}

\newcommand{\B}{\mathcal{B}}

\newcommand{\cO}{\mathcal{O }}

\newcommand{\cR}{\mathcal{R}}
\newcommand{\D}{\mathcal{D}}

\newcommand{\eps}{\varepsilon}

\newcommand{\Var}{\mathrm{Var}}

\newcommand{\la}{\langle}
\newcommand{\ra}{\rangle}

\newcommand{\EE}{\mathbf{E}}

\newcommand{\cP}{\mathcal{P}}

\begin{document}

\title[Gaussian fluctuations of polymer overlaps]{Gaussian fluctuations of replica overlap in directed polymers}
\author{Yu Gu, Tomasz Komorowski}

\address[Yu Gu]{Department of Mathematics, University of Maryland, College Park, MD 20742, USA}

\address[Tomasz Komorowski]{Institute of Mathematics, Polish Academy
  of Sciences, ul. \'{S}niadeckich 8, 00-656, Warsaw, Poland.
  Institute of Mathematics, UMCS, pl. Marii Curie-Sklodowskiej 1
20-031 Lublin}

\maketitle

\begin{abstract}
In this short note, we prove a central limit theorem for a type of replica overlap of the Brownian directed polymer in a Gaussian random environment, in the low temperature regime and in all dimensions. The proof relies on a   superconcentration result for the KPZ equation driven by a spatially mollified noise, which is inspired by the recent work of Chatterjee \cite{C1}.
\bigskip



\noindent \textsc{Keywords:} directed polymer, KPZ equation, superconcentration.

\end{abstract}
\maketitle

\section{Introduction}

\subsection{Main result}
Let $\eta(t,x)$ be a spacetime white noise on $\R_+\times \R^d$, 
and $\phi\in C_c^\infty(\R^d)$ be a non-negative, compactly
    supported, smooth function. Define the
generalized Gaussian process $\xi$ by 
\begin{equation}\label{e.defxi}
\xi(t,x)=\int_{\R^d} \phi(x-y)\eta(t,y)dy.
\end{equation}
So $\xi$ is white in time with the spatial covariance function 
\begin{equation}\label{e.defR}
R(x)=\int_{\R^d}\phi(x+y)\phi(y)dy,\quad  x\in\R^d.
\end{equation}
Let $B$ be a standard Brownian motion that is independent of $\eta$, starting from the origin. We assume that $B$ and $\eta$ are defined on a common probability space $(\Omega,\F,\Pb)$, and let $\E$ and $\EE$ denote the expectations with respect to $B$ and $\eta$ respectively.

The Brownian directed polymer in the random environment $\xi$ was
introduced in \cite{TR}. We briefly describe it as follows. For
each realization of the noise $\xi$, fixed $\beta>0$ and $T>0$, define the point-to-line polymer measure $\hat{\Pb}_T$ on $C[0,T]$ as  the Wiener measure tilted by the Radon-Nikodym derivative 
\[
\frac{e^{\beta \int_0^T \xi(s,B_s)ds-\frac12\beta^2R(0)T}}{Z_T},
\]
where $Z_T$ is the partition function 
\begin{equation}\label{e.defZ}
Z_T=\E [e^{\beta \int_0^T \xi(s,B_s)ds-\frac12\beta^2R(0)T}].
\end{equation}
The expectation with respect to the polymer measure $\hat{\Pb}_T$ is denoted by $\hat{\E}_{T}$. In other words, for any bounded   $F: C[0,T]\to\R$, we have 
\[
\hat{\E}_{T}[F(B)]=Z_T^{-1}\E[e^{\beta \int_0^T \xi(s,B_s)ds-\frac12\beta^2R(0)T}F(B)].
\]
For any $t\geq0$, define the overlap of the polymer endpoint at time $t$ as 
\[
\begin{aligned}
&\hat{\E}_t^{\otimes 2} [R(B_1(t)-B_2(t))]\\
&=Z_t^{-2} \E^{\otimes 2}[e^{\beta \int_0^t[\xi(s,B_1(s))+\xi(s,B_2(s))]ds-\beta^2R(0)t} R(B_1(t)-B_2(t))],
\end{aligned}
\]
where $B_1,B_2$ are two independent copies of Brownian motions, and $\hat{\E}_t^{\otimes 2}$ is the expectation with respect to $\hat{\Pb}_t^{\otimes 2}$. 

For each $T\geq0$, we define the replica overlap up to time $T$ as
\begin{equation}\label{e.defO}
\mathcal{O}_T=\int_0^T \hat{\E}_t^{\otimes 2} [R(B_1(t)-B_2(t))]dt,
\end{equation}
and we will consider the so-called low temperature regime. It is well-known that as $T\to\infty$,
\begin{equation}\label{e.confree}
\frac{1}{T}\log Z_T\to-\gamma(\beta)
\end{equation}
almost surely, where $\gamma(\beta)\geq 0$ is some constant, see
\cite[Proposition 2.6]{TR}. The {\em low temperature 
regime} is defined  as the set of those $\beta$ such that
$\gamma(\beta)>0$, see \cite[Definition 2.1, p. 27]{comets}. Note that the partition function defined in \eqref{e.defZ} is normalized so that $\EE Z_T\equiv1$, therefore the $\gamma(\beta)$  obtained above actually equals to the difference between the quenched and annealed free energy.

It is a popular topic in the study of directed polymers to
  consider different notions of strong and weak disorder regimes. For
  our model and under the assumption of $0\leq R(\cdot)\in L^1(\R^d)$,
  we expect that  the low temperature  regime is $\{\beta>0\}$ in
  $d=1,2$ and $\{\beta>\beta_c\}$ in $d\geq3$ for some critical
  $\beta_c>0$. Actually, it follows from \cite[Theorem 1.3]{lacoin}
  that the low temperature regime in $d=1$ is  $\{\beta>0\}$. In
  $d=2$, it was shown in \cite{lacoin1} for a discrete model that the
  low temperature regime  in $d=2$ is also $\{\beta>0\}$, so it is
  natural to expect that the same holds in our continuous setting, see a similar discussion in  \cite[Remark 1.5]{lacoin}. Since this is not the focus of this note, we do not attempt to follow the proof of \cite{lacoin1} in the discrete setting to establish this   for our model.  The phase transition in $d\geq3$ is well-known, see e.g. the work of \cite{ofer}.

By a semimartingale decomposition, see \eqref{ZM} and \eqref{MO} below, we have
\[
\EE \cO_T=-2\beta^{-2}\EE \log Z_T \approx 2\beta^{-2} \gamma(\beta)T, \quad\quad \mbox{ for } T\gg1.
\]
In other words, the mean of the replica overlap grows linearly in $T$, in the low temperature regime. Now we can state the main result, which is on the random fluctuations of $\cO_T$ around $\EE \cO_T$:

%
%

\begin{theorem}\label{t.mainth}
In the low temperature regime, we have 
\[
\frac{1}{\sqrt{T}}\big( \cO_T-\EE\cO_T\big)\Rightarrow N(0,8\gamma(\beta)\beta^{-4})
\]
in distribution, as $T\to\infty$. 
\end{theorem}

\subsection{Motivation}
\label{s.motivation}

The directed polymer in random environment is a popular subject in
probability and statistical physics, and a prototype model in the
study of  disordered physical systems. Here we will not attempt to review the large body of literature and only refer the readers to the monograph \cite{comets}, the introduction of  \cite{BC1} and the references therein. Our interest in the replica overlap defined in \eqref{e.defO} are twofold. 

(i) Quantities of the form \eqref{e.defO} are closely related to the
localization phenomenon, which has been extensively studied, see \cite{comets1,comets4,comets2,comets3,BC1,BC2,C2,Bates,vargas} and the references therein. In
the low temperature  regime, $\cO_T$ grows linearly with   $T$, which can already be taken as a sign of localization. It shows that, in a time averaged sense, $\hat{\E}_t^{\otimes 2} [R(B_1(t)-B_2(t))]$ is strictly positive, which implies that, since $R(\cdot)$ is fast-decaying, the endpoints of the two independent samples   from $\hat{\Pb}_t$ must be ``close to each other''. We refer to \cite[Chapter 5, page 76-77]{comets} for an interpretation of $\cO_T$ as a ``replica overlap''. 
 Another form of replica overlap may be defined as 
\begin{equation}\label{e.defO1}
\mathscr{O}_T=\int_0^T \hat{\E}_T^{\otimes 2} [R(B_1(t)-B_2(t))]dt,
\end{equation}
where the average is taken with respect to a fixed Gibbs measure and  
is arguably more natural. It is well-known that $\cO_T$ and
$\mathscr{O}_T$ appear in different contexts, one through It\^o
calculus and the other through Malliavin calculus, see \cite{comets1}
for a nice discussion. Under certain assumptions, one can also show
that  $\EE\mathscr{O}_T$ grows linearly with $T$, see
e.g. \cite[Proposition 2.3]{comets1} and \cite[Equation (1.7)]{Bates}
and the references cited there. See also \cite{BC2} for some relevant
result along the line of concentration of $\mathscr{O}_T/T$.  To us,
it seems very natural  to consider the next order fluctuations, beyond
the linear growth. The present note studies the fluctuations of
$\cO_T$, which turns out to be Gaussian. We are curious whether
the same holds for $\mathscr{O}_T$. We present  some further discussions on the implications of our result in Section~\ref{s.discussion}.

(ii) The free energy of the directed polymer is given by $\log Z_T$,
the fluctuations of which are expected to be sub-diffusive in all
dimensions. It is   related to the solution to the KPZ equation,
driven by $\xi$ and started from a constant initial data, see
\eqref{e.lawequal} below. In $d=1$, the fluctuation exponent for $\log Z_T$ is expected
to be $1/3$, which was proved
for several models in the 1+1 KPZ universality class, see e.g. \cite{timo,acq,bqs,bcf} and the reviews \cite{corwin2012kardar,qs}.
In dimensions higher than one, the exponent is unknown, while the
 variance is again expected to grow sublinearly, which is the so-called
 superconcentration in \cite{C}. We will show in
 Theorem~\ref{t.supercon191}  that the variance of $\log Z_T$,
   hence also the solution to
 the KPZ equation, behaves sublinearly, as $T\gg1$. Previous results on the same type of superconcentration can be found in the recent paper \cite{C1} and the references therein. To us, a somewhat natural way of deriving and quantifying the superconcentration phenomenon is to  express $\log Z_T$ using a semimartingale decomposition: since $Z$ itself is a positive martingale, we have $\log Z_T=M_T-\frac12\la M\ra_T$, 
where $M$ is a continuous martingale and $\la
M\ra$ is its quadratic variation. It turns out that the overlap
$\cO_T$ is just $\beta^{-2}\la M\ra_T$, see \eqref{e.defM} below.  A
simple argument invoking \eqref{e.confree} and the martingale central
limit theorem directly shows that $M_T/\sqrt{T}$ is asymptotically
Gaussian, in the low temperature regime. Therefore, the   central
limit theorem for $(\la M\ra_T-\EE \la M\ra_T)/\sqrt{T}$ is actually a
necessary condition for the superconcentration of $\log Z_T-\EE\log
Z_T$, and one would expect that a detailed understanding of the
Gaussianity coming  from $(\la M\ra_T-\EE \la M\ra_T)/\sqrt{T}$ could
help with quantifying the superconcentration phenomenon. This has been our
original motivation to study the fluctuations of $\cO_T=\beta^{-2} \la
M\ra_T$. It turns out that $\la M\ra_T$ can be written as an additive
functional of a Markov process $\{\rho(t,\cdot)\}_{t\geq0}$, which
takes values in the space of probability measures on $\R^d$. For each
$t\geq0$, $\rho(t,\cdot)$ is  the endpoint distribution of the polymer
path under $\hat{\Pb}_t$, an object that has been extensively studied. In \cite{gk}, we considered the
problem on a torus, and showed  that $\{\rho(t,\cdot)\}_{t\geq0}$ has
a unique invariant measure and   converges exponentially fast to
it in an appropriate Fortet-Mourier metric. Then by solving the Poisson equation corresponding to the
generator of the process and performing another martingale decomposition, we showed that $(\la M\ra_T-\EE\la M\ra_T)/\sqrt{T}$ satisfies a central limit theorem. Nevertheless, when it is on a torus, the variance of $\log Z_T$ grows linearly: $\Var \log Z_T \sim T$, so there is no complete cancellation between $M_T$ and $\frac12[\la M\ra_T-\EE \la M\ra_T]$. Some further attempts  have been made in \cite{dgk} to increase the size of the torus with time and to quantify the superconcentration phenomenon in $d=1$, leading to optimal exponents in certain regimes, without covering the case of the whole space though. We are curious if one can study the aforementioned additive functional directly, by establishing a certain mixing property of the process $\{\rho(t,\cdot)\}_{t\geq0}$.  At this point, it is worth mentioning the recent works of \cite{BC1,mukherjee,bakhtin1}, where the probability space is compactified to study the evolution of the process $\{\rho(t,\cdot)\}_{t\geq0}$.

As mentioned previously, the proof of Theorem~\ref{t.mainth} relies on proving the superconcentration of $\log Z_T$. Similar results have been obtained in \cite{AZ,graham,C} for different models. Our approach follows \cite{C1}, and a crucial input is an estimate on the spatial variations of the solution to the KPZ equation, see Proposition~\ref{p.bdburgers} below. This is a version of the ``subroughness'' defined in \cite{C1}, and provides an (sub-optimal) upper bound on the fluctuations of the spatial increments of the solution to the KPZ equation, see Remark~\ref{r.sub}. By the local averaging trick of Benjamini-Kalai-Schramm \cite{BKS}, the superconcentration follows from Talagrand's $L^1-L^2$ bound, see \cite{talagrand} and \cite[Chapter 5]{C}.

The rest of the note is organized as follows. In
Section~\ref{s.proof}, we prove the main result assuming the
superconcentration of $\log Z_T$, which is shown in
Section~\ref{s.supercon}. Some further discussions are carried out in Section~\ref{s.discussion}.

\subsection*{Acknowledgements}
We thank Erik Bates and Sourav Chatterjee for comments on the draft and two anonymous referees for a careful reading of the paper  which helps to improve the presentation. Y.G. was partially supported by the NSF through DMS-2203007/2203014.  T.K. acknowledges the support of NCN grant 2020/37/B/ST1/00426.

\section{Proof of the main result}\label{s.proof}

The partition function $Z_T$ defined in \eqref{e.defZ} is a positive martingale, and the following semi-martingale decomposition of $\log Z_T$ is well-known:
\begin{equation}
\label{ZM}
\log Z_T=M_T-\frac12\la M\ra_T.
\end{equation}
Here
\begin{equation}\label{e.defM}
\begin{aligned}
&M_T=\int_0^T Z_t^{-1}dZ_t=\beta\int_0^T\int_{\R^d} \rho(t,y)\xi(t,y)dydt,\\
&\la M\ra_T=\int_0^T Z_t^{-2} d\la Z\ra_t=\beta^2\int_0^T\int_{\R^{2d}} \rho(t,y)\rho(t,y')R(y-y')dydy'dt.
\end{aligned}
\end{equation}
The $\rho(t,\cdot)$ here is the endpoint density of the directed polymer under $\hat{\Pb}_t$, i.e., 
\begin{equation}\label{e.defrho}
\rho(t,x)=Z_t^{-1} \E[e^{\beta \int_0^t \xi(s,B_s)ds-\frac12\beta^2 R(0)t}\delta(B_t-x)].
\end{equation}
From \eqref{e.defO} and \eqref{e.defM}, we  know that
\begin{equation}
\label{MO}
\la M\ra_T=\beta^2\cO_T.
\end{equation}
The proof of Theorem~\ref{t.mainth} relies on the following lemmas.

\begin{lemma}\label{l.conma192}
$\frac{1}{T}M_T\to0$ in the $L^2$-sense, as $T\to\infty$.
\end{lemma}
\begin{proof}
We have $\EE M_T^2=\EE \la M\ra_T=-2\EE \log Z_T$, so by \cite[Proposition 2.5]{TR}, we have 
\[
\frac{1}{T^2} \EE M_T^2=-\frac{2}{T^2} \EE \log Z_T\to0
\]
which completes the proof.
\end{proof}

\begin{lemma}\label{l.wkma191}
In the low temperature regime, as $\eps\to0$, we have 
\[
(\eps M_{T/\eps^2})_{T\geq0} \Rightarrow (\sigma W_T)_{T\geq0}
\]
in $C[0,\infty)$ with $\sigma=\sqrt{2\gamma(\beta)}>0$, where $W$ is a standard Brownian motion.
\end{lemma}

\begin{proof}
Since $(\eps M_{T/\eps^2})_{T\geq0}$ is a family of continuous, square
integrable martingales, it suffices to consider the quadratic variation. We write it explicitly:
\[
\eps^2 \la M\ra_{T/\eps^2}=\eps^2\beta^2\int_0^{T/\eps^2} dt\int_{\R^{2d}} \rho(t,y)\rho(t,y')R(y-y')dydy'.
\]
First, by combining \eqref{e.confree} and Lemma~\ref{l.conma192}, we have 
\[
\frac{1}{T}\la M\ra_T=\frac{-2}{T}(\log Z_T-M_T)\to 2\gamma(\beta)
\]
in probability. Thus, we have the convergence of finite dimensional distributions of the process $(\eps^2\la M\ra_{T/\eps^2})_{T\geq0}$ as $\eps\to0$. It remains to show the tightness. For any $t\geq s$, we have 
\[
\eps^2[\la M\ra_{t/\eps^2}-\la M\ra_{s/\eps^2}]=\eps^2\beta^2\int_{s/\eps^2}^{t/\eps^2}d\ell\int_{\R^{2d}}\rho(\ell,y)\rho(\ell,y')R(y-y')dydy'.
\]
Since $R(x)\leq R(0)$, we  have 
\[
\eps^2[\la M\ra_{t/\eps^2}-\la M\ra_{s/\eps^2}] \leq \beta^2R(0) (t-s),
\]
which implies tightness, see e.g. \cite[Theorem VI.4.12, p. 358]{jacod-shiryaev}. The proof is complete.
\end{proof}

The following result plays a crucial role in establishing the Gaussian fluctuations of the replica-overlap.
\begin{theorem}\label{t.supercon191}
There exists $C>0$ such that
$$
\Var \log Z_T \leq \frac{CT}{\log T}\quad\mbox{ for }T\geq2.
$$
\end{theorem}


The proof of Theorem~\ref{t.supercon191} is presented in
Section~\ref{s.supercon}. We first use it to complete the proof of the main result.

\begin{proof}[Proof of Theorem~\ref{t.mainth}]
Thanks to \eqref{MO} and \eqref{ZM} we can write
\[
\begin{aligned}
\frac{1}{\sqrt{T}}\big( \cO_T-\EE \cO_T\big)=&\frac{1}{\beta^2\sqrt{T}}\big( \la M\ra_T-\EE \la M\ra_T\big)\\
=&\frac{-2}{\beta^2\sqrt{T}}\big( \log Z_T-\EE \log Z_T\big)+\frac{2}{\beta^2\sqrt{T}} M_T\\
=&I_1+I_2.
\end{aligned}
\]
By Theorem~\ref{t.supercon191}, we have $I_1\to0$ as $T\to\infty$. Applying Lemma~\ref{l.wkma191}, we have 
\[
I_2\Rightarrow  N(0,4\sigma^2\beta^{-4}),
\]
with $\sigma^2=2\gamma(\beta)$, which completes the proof.
\end{proof}

\section{Superconcentration of KPZ}
\label{s.supercon}

Suppose that $u$ solves the stochastic heat equation driven by $\xi$, starting from constant,
\begin{equation}
  \label{SHE}
\begin{aligned}
&\partial_t u=\frac12\Delta u+\beta u\xi, \quad \quad t>0,x\in\R^d,\\
&u(0,x)\equiv1,
\end{aligned}
\end{equation}
and define $h(t,x)=\log u(t,x)$, which solves the KPZ equation 
\begin{equation}\label{e.kpz}
\begin{aligned}
&\partial_t h= \frac12\Delta h+\frac12|\nabla h|^2 +\beta \xi-\frac12\beta^2 R(0), \quad\quad t>0,x\in\R^d,\\
&h(0,x)\equiv 0.
\end{aligned}
\end{equation}
Recall that $\xi$ is smooth in the spatial variable. Thus, for
  each $t>0$ and fixed realization of the noise, $u(t,\cdot)$ and
  $h(t,\cdot)$ are actually smooth functions, and the solutions here
  are understood as strong solutions. The product between $u$ and $\xi$ in \eqref{SHE} is in the It\^o sense. Since
$\xi$ is stationary and the initial data is constant, it is
straightforward to check that, for each $t>0$, $\{u(t,x)\}_{x\in\R^d}$
is a stationary random field. Using the Feynman-Kac
formula and the invariance of the law of $\xi$ under the time reversal
transformation and spatial shifts, we conclude that, for each $t>0,x\in\R^d$, 
\begin{equation}\label{e.lawequal}
u(t,x)\stackrel{\text{law}}{=}Z_t.
\end{equation}
Therefore,   Theorem~\ref{t.supercon191} is equivalent with  
\begin{equation}\label{e.superkpz}
\Var\, h(t,x)\leq \frac{Ct}{\log t}
\end{equation}
for some $C>0$ independent of $t\geq2$. The sublinear growth of the variance is called superconcentration \cite{C}, so our goal is to show that the height function, evolving according to the KPZ equation, superconcentrates. Our proof is inspired by the recent work of Chatterjee \cite{C1}, in which he made the crucial observation that the superconcentration is equivalent with what he called the ``subroughness''. 

\subsection{Talagrand's $L^1-L^2$ bound}

The first tool we need is the concentration inequality by
Talagrand. Recall that $\xi$ is constructed from the space-time white
noise $\eta$ through a spatial convolution \eqref{e.defxi},
where $\phi$ is a smooth kernel. Let $\D$ denote the Malliavin derivative with respect to $\eta$, and define $H=L^2(\R_+\times\R^d)$ and use $\la\cdot,\cdot\ra$ to denote its inner product. For smooth random variable $X$, which is measurable with respect to $(\eta(s,y))_{s\geq0,y\in\R^d}$, we write 
\[
\D X=(\D_{s,y}X)_{s\ge0,y\in\R^d},
\]
which is an $H-$valued random variable. For any $p\geq1$, we use $\|\cdot\|_p$ to represent the norm of $L^p(\Omega)$.

We will show that the Malliavin derivative of the KPZ solution $h(t,x)$ (or the free energy $\log Z_t$) is explicitly related to the polymer density, see \eqref{e.Dhtx} below. This is not surprising: for the discrete polymer model with the underlying random environment given by i.i.d. random variables on the lattice, the derivative of $\log Z_t$ with respect to the random variable at a given lattice point is precisely the probability of the polymer path passing through that point. The only reason we use the language of Malliavin calculus here is because our   random environment is constructed from the spacetime white noise. The usage will be minimal though -- besides the following proposition which has a well-known discrete counterpart, see \cite[Theorem 5.1]{C}, we only need the following fact in the proof of Lemma~\ref{l.varhM} below:  if $X=\int_0^\infty \int_{\R^d} f(s,y) \eta(s,y)dyds$ for some $f\in H$, then $\D_{s,y} X=f(s,y)$. For a detailed introduction to Malliavin calculus, we refer to \cite[Chapter 1]{nualart}.

\begin{proposition}\label{p.talagrand}
Assume $X$ is a smooth random variable measurable with respect to
$(\eta(s,y))_{s\geq0,y\in\R^d}$, and $A_{s,y}$ is a function such that
$$
\|\D_{s,y}X\|_2\leq A_{s,y}\quad\mbox{ for all }s\geq0,y\in\R^d.
$$
Then, we have 
\begin{equation}\label{e.varbd}
\Var X\leq C\int_0^\infty\int_{\R^d}  \frac{A_{s,y}^2}{1+\log
  \frac{A_{s,y}}{\|\D_{s,y}X\|_1}}  dyds,
\end{equation}
where $C>0$ is a universal constant.
\end{proposition}
%

\begin{proof}
First, we have the following variance representation (see e.g. \cite[Equation (4.6)]{DO})
\[
\Var X=\int_0^\infty e^{-t} \EE [\la \D X, \cP_t \D X\ra] dt,
\]
where $\cP_t$ is the Ornstein-Uhlenbeck semigroup, associated with $\eta$. Then we write the inner product explicitly and interchange the order of integration:
\[
 \EE \la \D X, \cP_t \D X\ra =\int_0^\infty \int_{\R^d} \EE [\D_{s,y}X
 \cP_t \D_{s,y}X] dyds .
 \]
 This leads to
 \[
 \Var X=\int_0^\infty \int_{\R^d}  \left(\int_0^\infty e^{-t}  \EE [\D_{s,y}X \cP_t \D_{s,y}X]  dt\right) dyds. 
 \]
 For each $y,s$, we claim that 
 \begin{equation}\label{e.101}
 \int_0^\infty e^{-t}  \EE [\D_{s,y}X \cP_t \D_{s,y}X]  dt\leq  C\frac{ A_{s,y}^2}{1+\log \frac{A_{s,y}}{\|\D_{s,y}X\|_1}},
 \end{equation}
 from which the conclusion of the proposition follows.
 The proof of \eqref{e.101} now follows verbatim \cite[Proof of Theorem 5.1]{C}.
\end{proof}

\subsection{Spatial increments of KPZ}

The goal of this section is to show the following version of
``subroughness'', which provides an upper bound on the spatial
variations of the height function $h(t,\cdot)$. Similar estimates have
been   derived in \cite[Lemma 5.3]{dcl}.
\begin{proposition}\label{p.bdburgers}
We have 
\[
\EE |h(t,x)-h(t,y)|^2 \leq \beta^2 R(0) |x-y|^2\quad\mbox{ for all }t>0, x,y\in\R^d.
\]
\end{proposition}

\begin{proof}
By the mild formulation of the KPZ equation \eqref{e.kpz}, we have
\[
\begin{aligned}
h(t,x)=&\frac12\int_0^t\int_{\R^d} q_{t-s}(x-y)|\nabla h(s,y)|^2 dyds\\
&+\beta\int_0^t\int_{\R^d} q_{t-s}(x-y)\xi(s,y)dyds-\frac12\beta^2R(0)t.
\end{aligned}
\]
Here $q_t(x)=(2\pi t)^{-d/2}e^{-|x|^2/(2t)}$ is the standard heat kernel.

Taking the expectation on both sides, we obtain 
\[
\begin{aligned}
 \EE h(t,x)=\frac12\int_0^t\int_{\R^d} q_{t-s}(x-y)\EE[|\nabla h(s,y)|^2] dyds-\frac12\beta^2R(0)t
\end{aligned}
\]
Since $h(s,\cdot)$ is stationary in the $x$ variable (for each $s\geq0$), we denote $f(s)=\EE[|\nabla h(s,y)|^2]$, then the above identity becomes 
\begin{equation}\label{e.1111}
\EE h(t,x)=\frac12\int_0^t f(s)ds-\frac12\beta^2R(0)t.
\end{equation}
  On the other hand, recalling \eqref{ZM}, we have 
\begin{equation}\label{e.1112}
\EE h(t,x)=\EE \log Z_t=-\frac12 \EE \la M\ra_t=-\frac12\beta^2\int_0^t g(s)ds,
\end{equation}
where
$$
g(s):=\int_{\R^{2d}} \EE[\rho(s,y)\rho(s,y')] R(y-y')dydy',
$$
which is a non-negative, continuous function, bounded by $R(0)$. Combining \eqref{e.1111} and \eqref{e.1112}, we have
\[
\beta^2R(0)t -\beta^2\int_0^t g(s)ds=\int_0^t f(s)ds,\quad\quad t\geq0,
\]
which implies that $f(t)=\beta^2[R(0)-g(t)]$. In particular, we have 
\[
0\leq f(t)\leq \beta^2R(0).
\] By the second moment bound on $\nabla h(t,\cdot)$, we have
\[
 \EE  |h(t,x)-h(t,y)|^2\leq \beta^2R(0) |x-y|^2.
\]
The proof is complete.
\end{proof}

\begin{remark}\label{r.sub}
The estimate derived in Proposition~\ref{p.bdburgers} is sub-optimal for $|x-y|\gg1$. For example, in $d=1$, it is expected that $\EE |h(t,x)-h(t,y)|^2 \sim |x-y|$ in the stationary regime: when $\xi$ is a $1+1$ spacetime white noise, the invariant measure for $h$ is a two-sided Brownian motion which attains such a bound. For a colored noise which decorrelates sufficiently rapidly, there is an interesting conjecture in \cite[Conjecture 3]{bakhtin} along the same line. The above proof does not exploit the spatial mixing property of $\nabla h(t,\cdot)$, thereby leads to a sub-optimal estimate.
\end{remark}

\subsection{The Benjamini-Kalai-Schramm trick}

In this section, we adapt the standard Benjamini-Kalai-Schramm trick \cite{BKS} to complete the proof of Theorem~\ref{t.supercon191}. From now on, we  abuse the notations and also let $\|\phi\|_\infty,\|\phi\|_1$ represent the $L^\infty(\R^d)$ and $L^1(\R^d)$ norm of $\phi$.

  Let $B_M=[-M,M]^d$ be the box centered at the origin with $M$ to be
  chosen later (eventually to be large). Define 
\[
h_M(t)=\frac{1}{|B_M|}\int_{B_M} h(t,x) dx,
\]
with $|B_M|=(2M)^d$.
To estimate $\Var \log Z_t=\Var\, h(t,0)$, we write $h(t,0)=h(t,0)-h_M(t)+h_M(t)$,
and use the estimate 
\begin{equation}\label{e.bdlogZt}
\Var \log Z_t\leq 2 \Var [h(t,0)-h_M(t)]+2\Var[h_M(t)].
\end{equation}
Then Theorem~\ref{t.supercon191} is a direct consequence of the following two lemmas.

\begin{lemma}\label{l.hdiff}
  We have
  $$
  \Var [h(t,0)-h_M(t)]\leq R(0)\beta^2d M^2\quad\mbox{for all } t,\,M>0.
  $$
 \end{lemma}
 
\begin{lemma}\label{l.varhM}
   We have
  $$
  \Var[h_M(t)] \leq \frac{2C\beta^2\|\phi\|_\infty
    \|\phi\|_{1} t}{2+\log \Big(2^d\|\phi\|_\infty\|\phi\|^{-1}_{1}\Big)+d \log M}\quad\mbox{for all } t>0,\,M\ge1,
  $$
 where $C$ is the constant appearing in Proposition \ref{p.talagrand}.
 \end{lemma}
 
 \begin{proof}[Proof of Theorem~\ref{t.supercon191}]
 It suffices to apply the above two lemmas in \eqref{e.bdlogZt} and pick $M=t^\alpha$, where $\alpha\in(0,1/2)$ can be arbitrary. 
 \end{proof}

 \begin{proof}[Proof of Lemma~\ref{l.hdiff}]
 First, because $h(t,\cdot)$ is stationary in the $x$ variable, we have $\EE [h(t,0)-h_M(t)]=0$. 
 Hence, by triangle inequality we have
 \begin{align*}
   \Var [h(t,0)-h_M(t)]&=\|h(t,0)-h_M(t)\|_2^2   \\
     &\leq \left(\frac{1}{|B_M|}\int_{B_M} \|h(t,0)-h(t,x)\|_2 dx\right)^2.
 \end{align*}
 Applying Proposition~\ref{p.bdburgers}, we have 
 \[
 \Var[h(t,0)-h_M(t)] \leq R(0)\beta^2\left(\frac{1}{|B_M|}\int_{B_M} |x| dx\right)^2 \leq R(0)\beta^2 dM^2.
 \]
 \end{proof}
 
 \begin{proof}[Proof of Lemma~\ref{l.varhM}]
 Recall that, by the Feynman-Kac formula we get the following
   representation for the solution of \eqref{SHE}:
   $$
u(t,x)=\E \Big[\exp\left\{\beta\int_0^t \xi(t-\ell,x+B_{\ell})d\ell-\frac12\beta^2R(0)t\right\} \Big].
$$   
We write the exponent in the above display explicitly:
\[
\begin{aligned}
\int_0^t \xi(t-\ell,x+B_{\ell}) d\ell
=\int_0^t\int_{\R^d} \phi(x+B_{t-\ell}-y)  \eta(\ell,y)dy  d\ell.
\end{aligned}
\] 
Fix a realization of the Brownian motion $B$,  we have 
\[
\D_{s,y} \left(\int_0^t \xi(t-\ell,x+B_{\ell}) d\ell\right) = \phi(x+B_{t-s}-y), \quad\quad s\in[0,t],y\in\R^d.
\]
 From here by a standard argument we get the formula for  the Malliavin
 derivative of $h(t,x)$, with respect to $\eta$: 
 \begin{equation}\label{e.Dhtx}
 \begin{aligned}
 \D_{s,y} h(t,x)=&\D_{s,y}\log u(t,x)=u(t,x)^{-1} \D_{s,y}u(t,x)\\
 =&\frac{\beta\E [e^{\beta\int_0^t \xi(t-\ell,x+B_{\ell})d\ell} \phi(x+B_{t-s}-y)]}{\E [e^{\beta\int_0^t \xi(t-\ell,x+B_{\ell})d\ell}]}.
 \end{aligned}
 \end{equation}
 From the above expression, it is clear that  
 \begin{equation}\label{e.bdmade}
 0\leq \D_{s,y}h(t,x)\leq\beta\|\phi\|_\infty,
\end{equation} and for all $t>0,x\in\R^d$ and $s\in[0,t]$, we have
\begin{equation}
  \label{011401-22}
  \int_{\R^d} \D_{s,y}h(t,x)dy=\beta\|\phi\|_{1}.
\end{equation}

 To apply Proposition~\ref{p.talagrand}, we first estimate 
 \[
 \D_{s,y}h_M(t)= |B_M|^{-1}\int_{B_M} \D_{s,y}h(t,x)dx.
 \] 
 By   the stationarity of $h(t,\cdot)$ in the spatial variable and \eqref{011401-22}, we have 
 \begin{equation}\label{e.bdDh1}
 \begin{aligned}
 \|\D_{s,y}h_M(t)\|_1= &\frac{1}{|B_M|} \int_{B_M} \|\D_{s,y}h(t,x)\|_1 dx\\
 =&\frac{1}{|B_M|}\int_{B_M} \|\D_{s,y-x}h(t,0)\|_1 dx \leq \frac{1}{|B_M|} \beta\|\phi\|_{L^1(\R^d)}.
 \end{aligned}
 \end{equation}

 For the $L^2(\Omega)$ norm,   we have 
 \[
 \EE |\D_{s,y}h_M(t)|^2\leq  \frac{1}{|B_M|}\int_{B_M} \EE |\D_{s,y}h(t,x)|^2 dx.
 \]
 By \eqref{e.bdmade}, we further derive 
 \begin{equation}
   \label{021401-22}
 \EE |\D_{s,y}h_M(t)|^2 \leq \frac{\beta\|\phi\|_\infty}{|B_M|}\int_{B_M}\EE \D_{s,y}h(t,x) dx.
 \end{equation}
 Let
 $$
 A_{s,y}:=\left\{\frac{\beta\|\phi\|_\infty}{|B_M|}\int_{B_M}\EE
   \D_{s,y}h(t,x) dx\right\}^{1/2}
 $$
By \eqref{021401-22} we have $\|\D_{s,y}h_M(t)\|_2\leq A_{s,y}$.


Applying Proposition~\ref{p.talagrand}, we have 
 \[
 \Var\, h_M(t) \leq C \int_0^t\int_{\R^d}  \frac{A_{s,y}^2}{1+\log
   \frac{A_{s,y}}{\|\D_{s,y}h_M(t)\|_1}} dy ds.
  \]
  By \eqref{e.bdDh1}, we have 
  \[
  \begin{aligned}
  \frac{A_{s,y}}{\|\D_{s,y}h_M(t)\|_1}&= \frac{A_{s,y}}{\frac{1}{|B_M|}\int_{B_M} \|\D_{s,y}h(t,x)\|_1 dx}\\
  &=\left\{\frac{\beta\|\phi\|_\infty |B_M|}{\int_{B_M} \|\D_{s,y}h(t,x)\|_1 dx}\right\}^{1/2} \geq \sqrt{\|\phi\|_\infty \|\phi\|^{-1}_{1}|B_M|}.
  \end{aligned}
  \]
  This, in turn implies 
  \[
  \Var\, h_M(t) \leq \frac{C }{1+\frac12\log \Big(\|\phi\|_\infty\|\phi\|^{-1}_{1}|B_M|\Big)}\int_0^t\int_{\R^d} A_{s,y}^2 dy ds.
  \]
  On the other hand, from the definition of $A_{s,y}$, we have 
  \[
  \begin{aligned}
  \int_0^t\int_{\R^d} A_{s,y}^2 dsdy=&\beta\|\phi\|_\infty |B_M|^{-1}\int_0^t\int_{\R^d}\left(\int_{B_M} \|\D_{s,y}h(t,x)\|_1 dx\right)dy ds\\
  =&\beta\|\phi\|_\infty |B_M|^{-1}\int_0^t\int_{\R^d}\left(\int_{B_M} \|\D_{s,y-x}h(t,0)\|_1 dx\right)dy ds \\
  =&\beta^2\|\phi\|_\infty \|\phi\|_{1}t,
  \end{aligned}
  \]
  where in the last ``='' we have used \eqref{011401-22}. The proof is complete.
 \end{proof}

 \section{Further discussion}
 \label{s.discussion}

 The approach here should also apply to other polymer models, including the ones in the discrete setting and the one with a $1+1$ spacetime white noise. A challenging problem is to study the other overlap $\mathscr{O}_T$ defined in \eqref{e.defO1}, and perhaps a  more modest question is actually to provide a different proof of Theorem~\ref{t.mainth}, without using the superconcentration of $\log Z_T$. In particular, one would like to understand that, in the following expression, 
 \[
 \begin{aligned}
 \cO_T
 =\int_0^T \cR(\rho(t,\cdot))dt, \quad  \mbox{ with } \mathcal{R}(f):=\int_{\R^{2d}} f(y)f(y')R(y-y')dydy',
 \end{aligned}
 \]
 where the mixing comes from and how it leads to the Gaussian fluctuations of $\cO_T-\EE \cO_T$. Recall that $\rho$ was defined in \eqref{e.defrho} and is the endpoint distribution of the directed polymer of length $t$. In a recent preprint \cite{das}, for the continuum directed polymer in the $1+1$ spacetime white noise, the following result was derived: for each $t>0$, the random density $\rho(t,\cdot)$ has a unique mode, denoted by $x_t$, and after a   shift by $x_t$, the following weak convergence on $C(\R)$ holds:  
 \begin{equation}\label{e.das}
 \{\rho(t,x_t+x)\}_{x\in\R} \Rightarrow  \left\{\frac{e^{-\B(x)}}{\int_{\R}e^{-\B(x')}dx'}\right\}_{x\in\R}, \quad\quad \mbox{ as } t\to\infty.
 \end{equation}
 Here $\B$ is a two-sided $3$d-Bessel process with diffusion coefficient $1$, see \cite[Theorem 1.5]{das} for more details. It is clear that $\mathcal{R}(\rho(t,\cdot))=\mathcal{R}(\rho(t,x_t+\cdot))$, so, in light of \eqref{e.das}, one may expect that $ \cR(\rho(t,\cdot))$ converges in distribution as $t\to\infty$, and for large $t$, $\cR(\rho(t,\cdot))$ mostly depends on the recent history of the random environment. This type of evidence of mixing is consistent with our result. 
 
On the other hand, 
we expect that 
 for any two initial distributions $\mu_j$, $j=1,2$ of densities $\rho_j(0,\cdot)$,
 the   respective processes $\rho_j(t,\cdot)$ satisfy
 \[
 \EE\int_{\R^d} |\rho_1(t,x)-\rho_2(t,x)|dx\to0,  \mbox{ as  }t\to\infty.
 \]
This is closely related to \cite[Theorem 4.4]{bakhtin2} which deals with a stationary version of the polymer measure in $1+1$ dimension. To study the mixing property of $\{\rho(t,\cdot)\}_{t\geq0}$, or more precisely, the randomly shifted one such as $\{\rho(t,x_t+\cdot)\}_{t\geq0}$ in \eqref{e.das}, or the overlap process $\{\cR(\rho(t,\cdot))\}_{t\geq0}$ which factors out the spatial shift, is an important question, the answer to which we believe is closely related to the localization behaviors of the polymer paths.


\end{document}